\documentclass[a4paper,11pt]{article}

\usepackage[top=3.0cm, bottom=3.0cm, inner=3.0cm, outer=3.0cm,
includefoot]{geometry}

\usepackage{appendix}
\usepackage{verbatim}
\usepackage{caption}
\usepackage{url}
\usepackage{amsmath}
\usepackage{geometry}
\usepackage{amssymb}
\usepackage{amsmath}
\usepackage{graphicx}
\usepackage{amsthm}
\usepackage{bbm}
\usepackage{float}
\usepackage{color,soul}
\usepackage{hyperref}
\usepackage[T1]{fontenc}
\usepackage[utf8]{inputenc}
\usepackage{authblk}
\usepackage{booktabs}
\usepackage{longtable}
\usepackage{enumerate}
\usepackage{tikz}
\usetikzlibrary{arrows}
\usetikzlibrary{decorations.markings}
\usepackage{indentfirst}
\newcommand{\eChar}{\begin{enumerate}[(i)]}
\newcommand{\eCharR}{\begin{enumerate}[(a)]}
\newcommand{\eBr}{\begin{enumerate}[(1)]}

\title
{
Halin graphs with positive Lin-Lu-Yau curvature
}
\author[1]{Kaizhe Chen\thanks{Email: ckz22000259@mail.ustc.edu.cn}}
\author[2]{Huiqiu Lin\thanks{Email: huiqiulin@126.com}}
\author[3]{Shiping Liu\thanks{Email: spliu@ustc.edu.cn}}
\author[2]{Zhe You\thanks{Email: y30231280@mail.ecust.edu.cn}}
\affil[1]{School of Gifted Young, University of Science and Technology of China}
\affil[2]{School of Mathematics, East China University of Science and Technology}
\affil[3]{School of Mathematical Sciences, University of Science and Technology of China}

\date{}

\theoremstyle{plain}
\newtheorem{lemma}{Lemma}[section]
\newtheorem{theorem}[lemma]{Theorem}

\newtheorem{corollary}[lemma]{Corollary}

\theoremstyle{definition}

\newtheorem{case}{Case}

\newtheorem{definition}[lemma]{Definition}
\newtheorem{subcase}{Subcase}[case]

\numberwithin{equation}{section}

\begin{document}

\maketitle

\begin{abstract}
Halin graphs constitute an interesting class of planar and polyhedral graphs. A generalized Halin graph is obtained by connecting all leaves of a planar embedding of a tree via a cycle. A Halin graph is a generalized Halin graph having no vertex of degree two.
We classify all generalized Halin graphs with positive Lin-Lu-Yau curvature. 
\end{abstract}

\section{Introduction}
%\textbf{Alternative way to write the introduction:}

It is natural to ask how large can a graph with given local structural constraints  be. Discrete notions of curvature are useful tools to describe such local structural constraints.

There has been an intensive study on graphs with positive combinatorial curvature. Combinatorial curvature is defined by angle deficiency at each vertex, which is a discrete analogue of Guassian curvature of surfaces. Notice that the definition of combinatorial curvature depends on the choice of embeddings of the graph into a surface.
%a discrete analogue of Gaussian curvature, the definition of which depends on the choices of embeddings of the graph into a surface. %Summing up the combinatorial curvature at every vertex yields Euler number, that is, a corresponding Gauss-Bonnet formula holds true. 
Higuchi \cite{Higuchi01} conjectured that any \emph{planar} graph with positive combinatorial curvature is finite. 
This conjecture has been verified for cubic planar graphs by Sun and Yu \cite{SY04}, and finally confirmed by DeVos and Mohar \cite{DM07}. Indeed, DeVos and Mohar showed a much stronger result: Let $G$ be a connected graph which is $2$-cell embedded into a surface $S$ so that every vertex and face has degree at least $3$. If the combinatorial curvature is positive at every vertex, then  $G$ is finite and $S$ is homeomorphic to either a $2$-sphere or the projective plane. Furthermore,
if $G$ is not a prism, antiprism, or the projective planar analogue of one of these, then the number of vertices $|V (G)|\leq 3444$. After Devos and Mohar's work, there has been a continuous effort in improving the upper bound $3444$ \cite{CC08,Oh17, RBK05,Zhang08}.  In \cite{NS11}, Nicholson and Sneddon  constructed graphs with positive combinatorial curvature of order $208$ embedded into a $2$-sphere. Recently, Ghidelli \cite{Ghidelli23} showed that $208$ is the optimal upper bound for the case of graphs embedded in a $2$-sphere and, for the case of  graphs embedded in a projective plane, the optimal upper bound is $104$.

Ricci curvature plays an important role in the study of Riemannian geometry and geometric analysis. It can be considered as an average of sectional (Gaussian) curvatures.  Recently, the study of discrete notions of Ricci curvature of graphs has attracted lots of attention, see \cite{NR17} and the references therein. Lin-Lu-Yau curvature \cite{LLY11}, which is defined via a modification of Ollivier's Ricci curvature \cite{Ollivier09}, has proved to be quite useful in detecting geometric, analytic and combinatorial properties of graphs. The Lin-Lu-Yau curvature is defined on each edge of a graph. The Lin-Lu-Yau curvature of an edge $\{x,y\}$ is positive if and only if the optimal transportation distance between certain neighborhoods of $x$ and $y$ is smaller than the combinatorial distance between $x$ and $y$. In particular, this definition does not rely on embeddings of the graph into a surface or higher dimensional differential manifold. 

In contrast to the case of combinatorial curvature, graphs with positive Lin-Lu-Yau curvature can be infinite. One class of examples is given by anti-trees with suitable growth rate \cite{DLMP20}.

Restricting to planar graphs, Lu and Wang \cite{LW20} established an analogue of DeVos and Mohar’s result with respect to Lin-Lu-Yau curvature. 
\begin{theorem}[Lu-Wang \cite{LW20}]
    Let $G=(V,E)$ be a connected planar graph such that every vertex has degree at least $3$. If the Lin-Lu-Yau curvature is positive at every edge,
 then $G$ is finite. In particular, $|V (G)| \leq 17^{544}$.
\end{theorem}
This upper bound is far from sharp. It seems to be difficult to figure out the optimal upper bound. Restricting further to outerplanar graphs, 
Liu, Lu and Wang \cite{LLW24} established the following optimal upper bound, improving a previous work \cite{BOSWY24}.  
\begin{theorem}[Liu-Lu-Wang \cite{LLW24}]
    Let $G=(V,E)$ be a connected outerplanar graph such that every vertex has degree as least $2$. If the Lin-Lu-Yau curvature is positive at every edge, then $|V(G)|\leq 10$ and the upper bound is optimal.
\end{theorem}
Moreover, Liu, Lu and Wang \cite{LLW24} classified all the outerplanar graphs with minimum degree at least 2 and positive Lin-Lu-Yau curvature using Sagemath and Nauty/geng.  

In this paper, we consider another interesting class of planar graphs called \emph{Halin graphs}. Such graphs were studied by Halin \cite{Halin71} as examples of minimally $3$-connected planar graphs. A Halin graph is a graph obtained from a planar embedding of a tree graph having at least $4$ vertices and no vertex of degree $2$ by connecting all the leaves of the tree with a cycle. 
A Halin graph is also known as a roofless polyhedron. In fact, a graph $G$ with $V(G)=n$ and $E(G)=m$ is Halin if and only if it is polyhedral (i.e., planar and $3$-connected) and has a face whose number of vertices equals $m-n+1$. 
Halin graphs have very interesting properties. For example,  they are $1$-Hamiltonian \cite{Bondy} and  almost pancyclic \cite{BL85}. %It is also an interesting topic to study the geometric properties of Halin graphs.

Indeed, we consider a larger class of graphs, which we call \emph{generalized Halin graphs}. In a generalized Halin graph, we allow vertices of degree $2$. A wheel graph is a typical example of Halin graphs. 
%directly, which also shows that $|V(G)|\leq 12$.
Let $W_n$ be a wheel graph with order $n$, i.e. the graph obtained by joining a vertex $v_0$ to all the vertices of a cycle $C_{n-1}=v_1\cdots v_{n-1}$. We further denote by $W_{n}'$ the graph obtained from $W_{n-1}$ by subdividing the edge $v_0v_1$, and by $W_{n}''$ the graph obtained from $W_{n-2}$ by subdividing the two edges $v_0v_1$ and $v_0v_{\lceil\frac{n-2}{2}\rceil}$. Both $W_n'$ and $W_n''$ are generalized Halin graphs but not Halin. 

We identify all generealized Halin graphs with positive Lin-Lu-Yau curvature.
\begin{theorem}\label{generalized Halin graph with positive LLY}
    Let $G$ be a generalized Halin graph with positive Lin-Lu-Yau curvature. Then $|V(G)|\leq 12$, and $G$ is isomorphic to one of the following graphs: $W_n (4\leq n\leq 12)$, $W_{n}' (5\leq n\leq 9)$, $W_{n}'' (6\leq n\leq 10)$, and $H_i (1\leq i\leq 8)$ depicted in Figure \ref{H_i}.
\end{theorem}
\begin{figure}[ht]
   \centering
   \includegraphics[width=0.95\textwidth,height=0.4\textwidth]{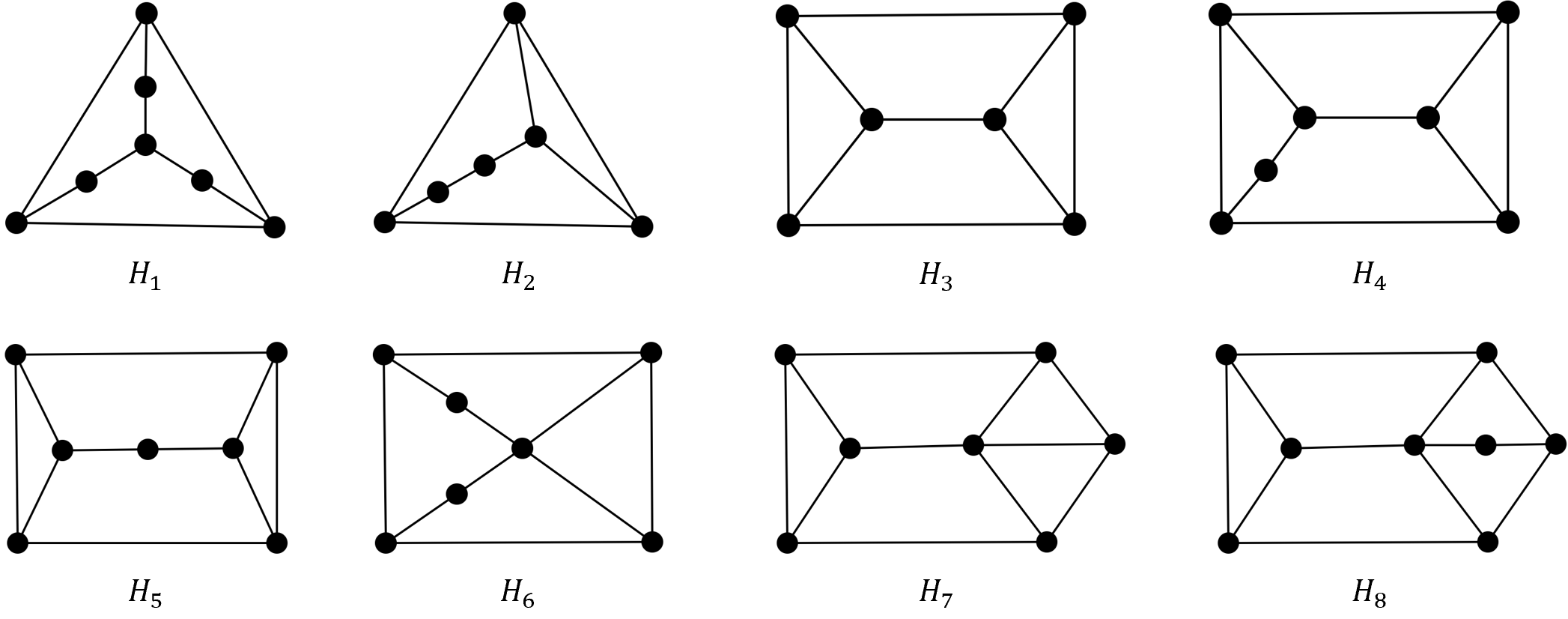}
   \caption{Generalized Halin graphs $H_i, 1\leq i\leq 8$}
  \label{H_i}
\end{figure}
%All the graphs in Figure \ref{H_i} have positive Lin-Lu-Yau curvature. We can calculate their Lin-Lu-Yau curvature by the graph curvature calculator \cite{CKLLS22}.

Consequently, we classify all Halin graphs with positive Lin-Lu-Yau curvature.
\begin{corollary}
    Let $G$ be a Halin graph with positive Lin-Lu-Yau curvature. Then $G$ is isomorphic to one of the graphs: wheel graphs $W_n (4\leq n\leq 12)$ and the graphs $H_i$ $(i=3,7)$ depicted in Figure \ref{H_i}.
\end{corollary}

%Comments on the proof tricks: key points: divided cases according to the degree of the central vertex?

A key step in our proof is to divide generalized Halin graphs with positive Lin-Lu-Yau curvature into subclasses, according to the maximum degree $D(T)$ of the corresponding tree graphs.

Throughout the paper, we use the following notation. Let $G=(V,E)$ be an undirected simple connected graph. For any $x\in V$, let $N(x)$ be the set of neighbors of $x$ and $d_x:=|N(x)|$ be its degree. We use $N[x]$ to denote $N(x)\cup \{x\}$. Let $D(G)$ be the maximum vertex degree of $G$. For any $x,y\in V(G)$, we denote the combinatorial distance between $x$ and $y$ by $d(x,y)$. For a vertex $x$ and a vertex set $S$, the distance between $x$ and $S$ is denoted by
$$d(x,S):=\min\{d(x,s)|s\in S\}.$$
We write $x\sim y$ if $\{x,y\}\in E$ is an edge. A function $f$ on $V(G)$ is called $1$-Lipschitz if 
$|f(x)-f(y)|\leq d(x,y)$ holds for any $x,y\in V(G)$, and we denote the set of $1$-Lipschitz functions by $Lip(1)$. 

\section{Preliminaries}\label{Halin Preliminaries}
In this sections, we collect some basics about Lin-Lu-Yau curvature and Halin graphs.

\subsection{Ollivier's Ricci curvature and Lin-Lu-Yau curvature}
We recall the definitions of Ollivier's Ricci curvature and Lin-Lu-Yau curvature. First, we recall the definition of Wasserstein distance between probability measures on graphs. 
\begin{definition}[Wasserstein distance]
Let $G=(V,E)$ be a locally finite graph with two probability measures $m_1$ and $m_2$ on $V$. The Wasserstein distance between $m_1$ and $m_2$ is defined as
$$W(m_1,m_2):=\inf\limits_\pi\sum\limits_{y\in V}\sum\limits_{x\in V}\pi(x,y)d(x,y),$$
where the infimum is taken over all maps $\pi:V\times V\rightarrow [0,1]$ satisfying
$$m_1(x)=\sum\limits_{y\in V}\pi(x,y) \text{ for any $x\in V$ and }m_2(y)=\sum\limits_{x\in V}\pi(x,y) \text{ for any $y\in V$}.$$
\end{definition}
%Based on the Kantorovich duality theorem, the Wasserstein distance can also be written as
%$$W(m_1, m_2)=\sup _f \sum_{x \in V} f(x)[m_1(x)-m_2(x)],$$
%where the supremum is taken over all $1$-Lipschitz functions $f$, which mean \[|f(x)-f(y)|\leq d(x,y), \quad \forall\,x,y\in V(G). \]
For a vertex $x\in V$ and any $\alpha\in[0,1]$, the probability measure $m_x^\alpha$ is defined as
     $$m_x^\alpha(v)= \begin{cases}\alpha, & \text { if } v=x, \\ (1-\alpha) / d_x, & \text { if } v \in N(x), \\ 0, & \text { otherwise. }\end{cases}$$
Then, the $\alpha$-Ollivier-Ricci curvature of an edge $\{x,y\}\in E(G)$ is defined as
$$\kappa_\alpha(x, y)=1-W\left(m_x^\alpha, m_y^\alpha\right).$$
In particular, we have $\kappa_\alpha(x,y)>0$ if and only if $W(m_x^\alpha, m_y^\alpha)<1=d(x,y)$.
%Lin, Lu and Yau \cite{LLY11} proved that $\frac{\kappa_\alpha(x,y)}{1-\alpha}$ is increasing on $[0,1)$ and bounded, which means that$\lim\limits_{\alpha\rightarrow 1}\frac{\kappa_\alpha(x,y)}{1-\alpha}$ exists. Moreover, they introduced a modified Ollivier-Ricci curvature as follows:
\begin{definition}[Lin-Lu-Yau curvature]
    Let $G=(V,E)$ be a locally finite graph. The Lin-Lu-Yau curvature of an edge $\{x,y\}\in E$ is defined as
    $$\kappa_{LLY}(x,y):=\lim\limits_{\alpha\rightarrow 1}\frac{\kappa_\alpha(x,y)}{1-\alpha}.$$
\end{definition}
%  We say a graph $G$ has positive Lin-Lu-Yau curvature if $\kappa_{LLY}(x,y)>0$ for any edge $\{x,y\}$. 
The ratio $\kappa_\alpha(x,y)/(1-\alpha)$ is constant when $\alpha$ is large enough. Indeed, it was proved in \cite{BCLMP18} that 
\begin{equation}\label{eq:bourne}
    \kappa_{LLY}(x,y)=\frac{\kappa_\alpha(x,y)}{1-\alpha}, \,\,\text{for any $\alpha\in \left[\frac{1}{\max\{d_x,d_y\}+1},1\right]$}.
\end{equation}
Another limit-free formulation of Lin-Lu-Yau curvature is due to M\"{u}nch and Wojciechowski \cite{MW19}. For any locally finite graph $G$, the  normalized graph Laplacian $\Delta$ is defined as
$$\Delta f(x):=\frac{1}{d_x} \sum_{y: y \sim x}(f(y)-f(x)), \text{ for any $f: V(G)\to \mathbb{R}$ and any $x\in V(G)$}.$$
The following identity is proved via a reformulation of the Kantorovich duality for Wasserstein distance.
\begin{theorem}[Curvature via the Laplacian {\cite[Corollary 2.2]{MW19}}]\label{Curvature via the Laplacian}
    Let $G$ be a locally finite graph and let $\{x,y\}$ be an edge. Then
    $$\kappa_{LLY}(x, y)=\inf _{\substack{f:N[x]\cup N[y]\to \mathbb{Z}\\f \in Lip(1) \\ f(y)-f(x)=1}} \left(\Delta f(x)-\Delta f(y)\right).$$
\end{theorem}
Next we recall a key lemma due to Lin, Lu and Yau \cite{LLY14}.
\begin{lemma}[{\cite[Lemma 2]{LLY14}}]\label{no C3 C4}
    Suppose that an edge $\{x,y\}$ in a graph $G$ is not in any $C_3$ or $C_4$. Then
$$\kappa_{LLY}(x,y)\leq\frac{1}{d_x}+\frac{2}{d_y}-1.$$
\end{lemma}
\begin{proof}
    As a warm-up, we give a proof of this lemma using Theorem \ref{Curvature via the Laplacian}. Consider the function $f\in Lip(1)$ defined as $f(z):=d(z,N[x]\backslash\{y\})$. We check that $f(y)-f(x)=1$. Then we have by Theorem \ref{Curvature via the Laplacian} 
    \begin{align*}
        \kappa_{LLY}(x,y)\leq \Delta f(x)-\Delta f(y)\le \frac{1}{d_x}-\frac{1}{d_y}\left(-1+\sum_{z:z\sim y, z\neq x}(2-1)\right)=\frac{1}{d_x}+\frac{2}{d_y}-1.
    \end{align*}
    This completes the proof.
\end{proof}

\subsection{Generalized Halin graphs}

\begin{definition}
    A generalized Halin graph $H(T,C)$ is a graph constructed by starting from a planar embedding of a tree $T$ with maximum vertex degree $D(T)\geq 3$, and connecting all the leaves of the tree with a cycle $C$. A Halin graph is a generalized Halin graph with no vertex of degree two.
\end{definition}
We remark that a tree $T$ may not decide a unique generalized Halin graph since it can be embedded into a plane in different ways.

\section{Proof of Theorem \ref{generalized Halin graph with positive LLY}}\label{section 3}
For convenience, we first introduce some notations. Let $G=H(T,C)$ be a generalized Halin graph. We denote by $d_T(\cdot,\cdot)$ the combinatorial distance in the tree $T$. Let $x$ be a vertex in $T$ with maximum degree $D(T)$. Recall that in a generalized Halin graph $G=H(T,C)$, we have a planar embedding of the tree $T$. We consider the connected components of $T-\{x\}$, and label them clockwise as $A_1, A_2,\ldots, A_c$. Furthermore, we label the leaves of $T$ in each $A_i$ clockwise as $A_{i,1},A_{i,2},\cdots,A_{i,k_i}$.  In this way, we label all the vertices on the cycle $C$ of $G=H(T,C)$, which we call {\it outer vertices}. We depict an example in Figure \ref{fig:clockwise}.

\begin{figure}[ht]
   \centering
   \includegraphics[width=0.4\textwidth,]{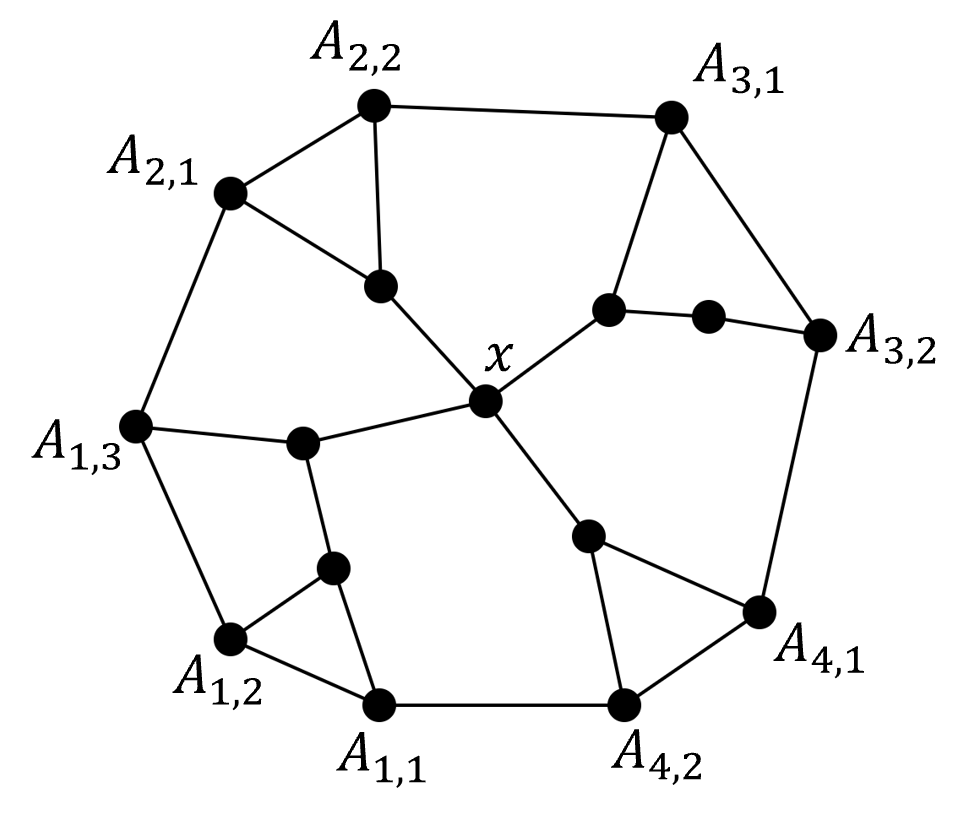}
   \caption{A example of labelling outer vertices in each component of $T-\{x\}$}
   \label{fig:clockwise}
\end{figure}

Before presenting our proof, we prepare two lemmas. 
\begin{lemma}\label{lemma: no double X}
Let $G=H(T,C)$ be a generalized Halin graph with positive Lin-Lu-Yau curvature, and $A_1,A_2,\ldots, A_c$ be defined as above. If $|A_i\cap C|\geq 2$ for some $i\in \{1,2,\ldots, c\}$, then $|A_{i+1}\cap C|=|A_{i-1}\cap C|=1$, where we use the notations $A_0=A_c$ and $A_{c+1}=A_1$.
  %  Suppose that $A_{i+1,1}$ and $A_{i,k}$ are adjacent. If $|A_{i+1}|\geq 2$, then $k=1$.
\end{lemma}
\begin{proof}
Suppose that either $|A_{i+1}\cap C|>1$ or $|A_{i-1}\cap C|>1$. Without loss of generality, we assume $|A_{i-1}\cap C|>1$. Consider the edge $\{A_{i,1}, A_{i-1,k_{i-1}}\}$. There exist vertices $A_{i,2}\sim A_{i,1}$ and $A_{i-1, k_{i-1}-1}\sim A_{i-1,k_{i-1}}$. 
Since $d_x=|D(T)|\ge 3$, we have $d(A_{i,2}, A_{i-1,k_{i-1}})\geq 2$. Therefore, we observe that the edge $\{A_{i,1}, A_{i-1,k_{i-1}}\}$ is not in any $C_3$ or $C_4$. By Lemma \ref{no C3 C4}, we have $\kappa_{LLY}(A_{i,1},A_{i-1,k_{i-1}})\leq \frac{1}{3}+\frac{2}{3}-1= 0$, which is a contradiction.
\end{proof}

\begin{lemma}\label{lemma:cycle<6}
Let $G=H(T,C)$ be a generalized Halin graph with positive Lin-Lu-Yau curvature, and $A_1,A_2,\ldots, A_c$ be defined as above.
For any two adjacent vertices $p$ and $q$ on $C$ which belong to different components $A_i$ and $A_j$, respectively, we have \[d_T(x,p)+d_T(x,q)\leq 4.\]
\end{lemma}
\begin{proof}
    Without loss of generality, we assume $d_T(x,p)\geq d_T(x,q)$. Suppose that $d_T(x,p)+d_T(x,q)\geq 5$.  Let $y$ be the neighbor of $x$ in the component $A_i$ and $t$ be the neighbor of $x$ in the component $A_j$.
Consider a function $f:N[x]\cup N[y]\rightarrow\mathbb{Z}$ given by
%$$f(z)= \begin{cases}
%-1, & \text { if } z=t ; \\
%0, & \text { if } z\in N[x]\backslash\{t,y\}; \\
%1, &  \text{ if } z=y;\\
%1, & \text{ if $z\in N(y)\backslash\{x\}$ and %$d(z,N(x)\cup\{p\}\backslash\{t,y\})=1$;}\\
%2, & \text { if $z\in N(y)\backslash\{x\}$ and $d(z,N(x)\cup\{p\}\backslash\{t,y\})>1$.}
%\end{cases}$$
%Observe that the assumption $d_T(x,p)+d_T(x,q)\geq 5$ ensures that $d(z,t)\geq3$ for all $z\in N(y)\backslash\{x\}$. 
$$f(z)=\min\{d(z,N[x]\backslash \{y\}),d(z,t)-1\}.$$
Since the minimum of two 1-Lipschitz functions is still 1-Lipschitz, the function $f$ is 1-Lipschitz.
%Observe that the function $f$ is 1-Lipschitz and $f(y)-f(x)=1$. Notice that there are at most one $z\in N(y)\backslash\{x\}$ satisfying $d(z,N(x)\backslash\{t,y\})=1$ and there are at most one $z'\in N(y)\backslash\{x\}$ satisfying $d(z',p)=1$. If both $z$ and $z'$ exist and $z\neq z'$, then we have $d_y\geq 4$. Hence, we derive by Theorem \ref{Curvature via the Laplacian}
By the definition of $f$, we have $f(y)=1$, $f(t)=-1$ and $f(z)=0$ for any $z\in N[x]\backslash \{y,t\}$.
It follows that $f(y)-f(x)=1$ and $\Delta f(x)=0$.
Let $s$ be the neighbor of $y$ on the unique path in $T$ connecting $y$ and $p$. Since $d_T(x,p)+d_T(x,q) \geq 5$, we have $f(s)=2$. Furthermore, $f(z)\ge 1$ for any $z\in N[y]\backslash \{x\}$. Thus, we derive $\Delta f(y)\ge 0$. It follows by Theorem \ref{Curvature via the Laplacian} that
\[\kappa_{LLY}(x,y)\le\Delta f(x)-\Delta f(y)\leq 0,\]
%\[\kappa_{LLy}(x,y)\leq -\frac{d_y-4}{d_y}\leq 0.\]
%If $z=z'$ exists, then we have $d_y\geq 3$. Therefore, Theorem \ref{Curvature via the Laplacian} leads to 
%\[\kappa_{LLY}(x,y)\leq -\frac{d_y-3}{d_y}\leq 0.\] Otherwise, if both $z$ and $z'$ do not exist, we have $d_y\geq 2$. Hence Theorem \ref{Curvature via the Laplacian} yields
%$$\kappa_{LLY}(x,y)\leq -\frac{d_y-2}{d_y}\leq 0.$$
which leads to a contradiction. Therefore, we have shown $d_T(x,p)+d_T(x,q)\leq 4$.
\end{proof}
\begin{corollary}\label{corollary}
Let $G=H(T,C)$ be a generalized Halin graph with positive Lin-Lu-Yau curvature. Let $x$ be a vertex of $T$ with maximum degree $D(T)$. If $T$ has exactly $D(T)$ leaves and $D(T)\geq 4$, then $d_T(x,y)\leq 2$ for any $y\in C$.
\end{corollary}
\begin{proof}
    Suppose that there exists a vertex $y\in C$ such that $d_T(x,y)\geq3$. Due to Lemma \ref{lemma:cycle<6}, $d_T(x,y)=3$ and $d_T(x,y_1)=d_T(x,y_2)=1$, where  $y_1,y_2$ are the neighbors of $y$ on $C$. Denote the neighbor of $x$ which is in the same component with $y$ by $x'$. Then, the edge $\{x,x'\}$ is shared by two 5-faces, which leads to $\kappa_{LLY}(x,x')\leq \frac{1}{2}+\frac{2}{D(T)}-1\leq 0$ by Lemma \ref{no C3 C4}. Hence, $d_T(x,y)\leq 2$ for any $y\in C$.
\end{proof}
\begin{proof}[Proof of Theorem \ref{generalized Halin graph with positive LLY}]
    Let $G=H(T,C)$ be a generalized Halin graph with positive Lin-Lu-Yau curvature generated by the tree $T$ and the cycle $C$. Suppose that $x$ is a vertex with maximum degree $D(T)$ in $T$. 
    According to the maximum degree $D(T)$ of $T$, we divide our proof into the following three cases.
   \begin{case}
    	$D(T)\geq 5$. Notice that $T$ has at least $D(T)$ leaves. We first show that $T$ actually has precisely $D(T)$ leaves. Suppose that $T$ has more than $D(T)$ leaves, then there exists a component $A_i$ of $T-\{x\}$ containing at least two outer vertices. Let $b\ne A_{i,2}$ be the other neighbor of $A_{i,1}$ on the cycle $C$, and $p\ne A_{i,1}$ be the other neighbor of $b$ on the cycle $C$. Consider a function $f_1:N[A_{i,1}]\cup N[b]\rightarrow\mathbb{Z}$ given by
$$f_1(z)= \begin{cases}
-1, & \text { if } z=A_{i,2} ; \\
0, & \text { if } z\in N[A_{i,1}]\backslash\{b,A_{i,2}\} ; \\
1, & \text { if } z\in N[b]\backslash\{A_{i,1},p\} ; \\
2, & \text { if } z=p.
\end{cases}$$
By assumption, $T$ has at least $6$ leaves. This ensures $d(A_{i,2},p)=3$, and hence $f_1\in Lip(1)$. Since $f(b)-f(A_{i,1})=1$, we have by Theorem \ref{Curvature via the Laplacian}
$$\kappa_{LLY}(A_{i,1},b)\leq \Delta f_1(A_{i,1})-\Delta f_1(b)=0,$$
which is a contradiction. Therefore, $T$ has exactly $D(T)$ leaves. 

Applying Corollary \ref{corollary}, we conclude that $d_T(x,y)\leq 2$ for any $y$ on the cycle $C$.

Let $a$ and $b$ be two vertices on $C$ with $d_T(x,a)=d_T(x,b)= 2$. If $a$ and $b$ are adjacent, then $\kappa_{LLY}(a,b)\leq0$ by Lemma \ref{no C3 C4} since the edge $\{a,b\}$ is not in any $C_3$ or $C_4$ when $D(T)\geq 5$.
Therefore, two vertices $a$ and $b$ on $C$ with $d_T(x,a)=d_T(x,b)= 2$ can not be adjacent. We divide our discussion into two subcases as follows.
\begin{subcase}
 There is no vertex $y$ on $C$ with $d_T(x,y)=2$. 
 
 Then $G$ must be a wheel graph. Select any vertex $y$ on $C$. Consider a function $f_2:N[x]\cup N[y]\rightarrow\mathbb{Z}$ given by
$$f_2(z)= \begin{cases}
1, & \text { if } z\in (N(x)\cap N(y))\cup\{y\} ; \\
0, & \text { if } z\in N(t)\backslash\{y\}\ {\rm for\ some}\ t\in N(x)\cap N(y) ; \\
-1, &  \text{ others}.
\end{cases}$$
Observe that $f_2(y)-f_2(x)=1$ and $f_2\in Lip(1)$. By Theorem \ref{Curvature via the Laplacian}, $\kappa_{LLY}(x,y)\leq \frac{8}{D(T)}-\frac{2}{3}$. When $D(T)\geq 12$, $\kappa_{LLY}(x,y)\leq0$. It is direct to check that $W_n(4\leq n\leq12)$ has positive Lin-Lu-Yau curvature by \eqref{eq:bourne}. As an example, we verify that $W_5$ has positive Lin-Lu-Yau curvature in Section \ref{section 4}. One can also check the curvature by the graph curvature calculator \cite{CKLLS22}, which is a freely accessible interactive app at 
\begin{center}
    https://www.mas.ncl.ac.uk/graph-curvature/
\end{center}
\end{subcase}
\begin{subcase}
There exists at least one vertex $p$ on $C$ that satisfies $d_T(x,p)=2$. 

Let $y$ be the neighbor of $x$ in the component corresponding to $p$.
Consider a function $f_3:N[x]\cup N[y]\rightarrow\mathbb{Z}$ given by
$$f_3(z)= \begin{cases}
1, & \text { if } z\in\{y,p\} ; \\
0, & \text { if } z\in (N(p)\cup\{x\})\backslash\{y\}; \\
-1, &  \text{ others}.
\end{cases}$$
Then, we check that $f_3(y)-f_3(x)=1$ and $f_3\in Lip(1)$. By Theorem \ref{Curvature via the Laplacian}, $\kappa_{LLY}(x,y)\leq \frac{4}{D(T)}-\frac{1}{2}$. If $D(T)\geq 8$, then $\kappa_{LLY}(x,y)\leq 0$. Thus, we only need to check the situations when $D(T)=5,6,7$. 

Recall that any two vertices $a$ and $b$ on the cycle $C$ with $d_T(x,a)=d_T(x,b)= 2$ can not be adjacent. The remaining situations are not so much. We list all of them in Figure \ref{graph subcase 1.2} in the appendix and find that only $W_7',W_8',W_9',W_8'',W_9'',W_{10}''$  have positive Lin-Lu-Yau curvature by using the graph curvature calculator.
\end{subcase}
\end{case}

\begin{case}
    $D(T)=3$. According to Lemma \ref{lemma: no double X}, we divide this case into two subcases.
    \begin{subcase}
        Each component of $T-\{x\}$ has only one outer vertex ($|V(C)|=3$).
        
Denote the three outer vertices  by $p_1$, $p_2$ and $p_3$. Consider a disordered triplet:$$(d_T(x,p_1),d_T(x,p_2),d_T(x,p_3)).$$
Due to Lemma \ref{lemma:cycle<6}, we can list all the possible triplets as follows:
$$(3,1,1)\quad(2,2,2)\quad(2,2,1)\quad(2,1,1)\quad(1,1,1).$$
Using the graph curvature calculator, we check that all the graphs corresponding to the above triplets have positive Lin-Lu-Yau curvature. This means that
$G\cong H_2,H_1,W_6'',W_5'$ or $W_4$ in this subcase.
    \end{subcase}
    \begin{subcase}\label{subcase 2.2}
        There is exactly  one component of $T-\{x\}$ having at least two outer vertices.
        
   We denote this component by $A_1$. If $A_1$ has at least three outer vertices, then there must exist an edge $\{p,q\}$ on $C$ with $d_T(x,p)=3$ and $d_T(x,q)=1$ such that $p\in A_1$ and $q$ is in another component by Lemma \ref{lemma:cycle<6}. Clearly, $\{p,q\}$ is not in any $C_3$ or $C_4$. We have $\kappa_{LLY}(p,q)\leq \frac{1}{3}+\frac{2}{3}-1=0$ by Lemma \ref{no C3 C4}. Therefore, $A_1$ has exactly two outer vertices. 

 Combining with Lemma \ref{lemma:cycle<6}, the remaining graphs are shown in Figure \ref{graph subcase 2.2} in the appendix, and it tells us $G\cong H_3, H_4$ or $H_5$ in this subcase.
 %Choose two adjacent vertices $s$ and $t$ on $C$ in different components satisfying $d_T(x,s)\geq d_T(x,t)$. It is easy to check the remaining situations by Lemma \ref{lemma:cycle<6} and graph curvature calculator.If $d_T(x,s)=3$ and $d_T(x,t)=1$, then $s\in T_1$. Hence $G\cong H_4$ or $H_5$.If $d_T(x,s)=d_T(x,t)=2$, then $G\cong H_4$.If $d_T(x,s)=2$ and $d_T(x,t)=1$, then $G\cong H_3$ or $H_4$.
    \end{subcase}
 \end{case}	
 \begin{case}
      $D(T)=4$.  We again divide this case into two subcases.
    \begin{subcase}
         $|V(C)|=4$.

Denote the cycle by $C=q_1q_2q_3q_4$. %$4$ outer vertices  by $q_1$, $q_2$, $q_3$ and $q_4$, where $q_2$ is not adjacent to $q_4$ and $q_1$ is not adjacent to $q_3$.
 Consider an ordered quadruple:
$$(d_T(x,q_1),d_T(x,q_2),d_T(x,q_3),d_T(x,q_4)).$$
Due to Corollary \ref{corollary}, we know that $d_T(x,q_i)\leq 2$.
Combining with Lemma \ref{lemma:cycle<6}, we  list all the possible quadruples as follows (we identify those ordered quadruples correponding to isomorphic graphs):
$$(2,2,2,2)\quad(2,2,2,1)\quad(2,2,1,1)\quad(2,1,2,1)\quad(2,1,1,1)\quad(1,1,1,1).$$
Actually, both of $(2,2,2,2)$ and $(2,2,2,1)$ lead to at least one edge $\{x,x'\}$ shared by two 5-faces. By Lemma \ref{no C3 C4}, $\kappa_{LLY}(x,x')\leq 0$. We directly verify  that all the remaining four situations have positive Lin-Lu-Yau curvature by the graph curvature calculator. Therefore, $G\cong W_5,W_6',W_7''$ or $H_6$ in this subcase.
    \end{subcase}
    \begin{subcase}
        $|V(C)|\geq 5$.

        Choose a component of $T-\{x\}$ with at least two outer vertices and denote it by $A_1$. First, we prove that the neighbor $y\in A_1$ of $x$ is of degree $3$. Suppose not, then either $d_y=2$ or $d_y=4$.

        If $d_y=2$, then $\{x,y\}$ must be shared by two 5-faces due to Lemma \ref{lemma:cycle<6}. In particular, $\{x,y\}$ is not in any $C_3$ or $C_4$. By Lemma \ref{no C3 C4}, we have $\kappa_{LLY}(x,y)\leq 0$, a contradiction.

        If $d_y=4$, then we consider a 1-Lipschitz function $f_4:N[x]\cup N[y]\rightarrow\mathbb{Z}$ given by
$$f_4(z)= \begin{cases}
-1, & \text { if } z=x_2 ; \\
0, & \text { if } z\in \{x,x_1,x_3\}; \\
1, &  \text{ if } z\in\{y,y_1,y_3\};\\
2, & \text { if } z=y_2.
\end{cases}$$
Here, $y,x_1,x_2,x_3$ are the neighbors of $x$ in clockwise order and $x,y_1,y_2,y_3$ are the neighbors of $y$ in clockwise order, according to the planar embedding of $T$.  Observe that $d(x_2,y_2)=3$. Hence, $f_4\in Lip(1)$. Noticing that $f_4(y)-f_4(x)=1$, we have  $\kappa_{LLY}(x,y)\leq 0$ by Theorem \ref{Curvature via the Laplacian}. This proves $d_y=3$. 

A similar argument as in Subcase \ref{subcase 2.2} tells us that $A_1$ has exactly two outer vertices. Let us again denote the neighbors of $x$ in clockwise order by $y, x_1, x_2$ and $x_3$ and denote the neighbors of $y$ in clockwise order by $x, y_1$ and $y_2$. Lemma \ref{lemma: no double X} tells us that the component containing $x_1$ or $x_3$ has precisely one outer vertex. If the component corresponding to $x_2$ has two outer vertices, then either there exists an edge $\{s_1,s_2\}$ on $C$ which is not in any $C_3$ or $C_4$ or $G\cong H'$, which is depicted in Figure \ref{H'}. In the first case, we have  $\kappa_{LLY}(s_1,s_2)\leq 0$ due to Lemma \ref{no C3 C4}. In the later case, the graph $H'$ has edges of $0$ curvature as shown in Figure \ref{H'}.

\begin{figure}[ht]
   \centering
\includegraphics[width=0.4\textwidth]{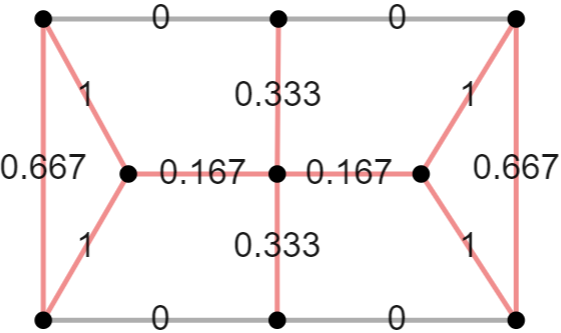}
   \caption{Graph $H'$}
   \label{H'}
\end{figure}

Thus, $A_1$ is the only component which has at least two outer vertices. To avoid the appearance of an edge $\{s_1,s_2\}$ on $C$ not containing in any $C_3$ or $C_4$, the vertices $x_1$, $x_3$, $y_1$ and $y_2$ must be the outer vertices. Moreover, $d_T(x,p)\leq 2$ where $p$ is the outer vertex in the component containing $x_2$. Otherwise, $\{x,x_2\}$ is shared by two big faces (of size at least 5), which leads to $\kappa_{LLY}(x,x_2)\leq\frac{1}{2}+\frac{2}{4}-1=0$ by Lemma \ref{no C3 C4}.
There are only two remaining graphs with $d_T(x,p)=1$ or $2$, respectively, which are $H_7$ and $H_8$ in Figure \ref{H_i}. We check that both of them have positive Lin-Lu-Yau curvature by the graph curvature calculator. Therefore, we have $G\cong H_7$ or $H_8$ in this subcase.
    \end{subcase}
    \end{case}	
This completes our proof.
\end{proof}

\section{An example}\label{section 4}
We only calculate the upper bounds of Lin-Lu-Yau curvature in our proof, and utilize the graph curvature calculator to verify whether a graph has positive Lin-Lu-Yau curvature or not for convenience. Below we take $W_5$ as an example to show how to  establish lower bounds for the Lin-Lu-Yau curvature via \eqref{eq:bourne}.

Denote the center vertex of $W_5$ by $x$ and the cycle by $C=x_1x_2x_3x_4$. By symmetry, we only need to consider $\kappa_{LLY}(x,x_1)$ and $\kappa_{LLY}(x_1,x_2)$. For $\frac{1}{4}\leq \alpha\leq 1$, consider a map $\pi_1:V\times V\rightarrow [0,1]$ defined as 
$$\pi_1(u,v)= \begin{cases}
\min\{m_x^\alpha(u),m_{x_1}^\alpha(u) \}, & \text { if } u=v ; \\
\alpha -\frac{1-\alpha}{3}, & \text { if } u=x,v=x_1 ; \\
\frac{1-\alpha}{3}-\frac{1-\alpha}{4}, & \text { if } u=x_3,v\in \{x_1,x_2,x_4\} ; \\
0, & \text { otherwise.}
\end{cases}$$
By (\ref{eq:bourne}), we derive
$$\kappa_{LLY}(x,x_1)=\frac{\kappa_\alpha(x,x_1)}{1-\alpha}\geq\frac{1-\sum\limits_{v\in V}\sum\limits_{u\in V} \pi_1(u,v)d(u,v)}{1-\alpha}=1.$$
Similarly, for $\frac{1}{4}\leq \alpha\leq 1$, consider another map $\pi_2:V\times V\rightarrow [0,1]$ defined as 
$$\pi_2(u,v)= \begin{cases}
\min\{m_{x_1}^\alpha(u),m_{x_2}^\alpha(u) \}, & \text { if } u=v ; \\
\alpha -\frac{1-\alpha}{3}, & \text { if } u=x_1,v=x_2 ; \\
\frac{1-\alpha}{3}, & \text { if } u=x_4,v=x_3 ; \\
0, & \text { otherwise.}
\end{cases}$$
By (\ref{eq:bourne}), we deduce
$$\kappa_{LLY}(x_1,x_2)=\frac{\kappa_\alpha(x_1,x_2)}{1-\alpha}\geq\frac{1-\sum\limits_{v\in V}\sum\limits_{u\in V} \pi_2(u,v)d(u,v)}{1-\alpha}=1.$$

Therefore, $W_5$ has positive Lin-Lu-Yau curvature.
\section*{Acknowledgement}
\noindent 
This work is supported by the National Key R \& D Program of China 2023YFA1010200. K. C. is supported by the New Lotus Scholars Program PB22000259. H. L. is supported by the National Natural Science Foundation of China (No. 12271162), Natural Science Foundation of Shanghai (No. 22ZR1416300 and No. 23JC1401500) and The Program for Professor of Special Appointment (Eastern Scholar) at Shanghai Institutions of Higher Learning (No. TP2022031). S. L. is supported by the National Natural Science Foundation of China No. 12031017 and No. 12431004.

\appendix
\section{Appendix}
We use the graph curvature calculator to draw the remaining graphs in each subcases of our proof. The number on each edge is the Lin-Lu-Yau curvature.
\begin{figure}[ht]
  \centering
 \includegraphics[width=0.8\textwidth]{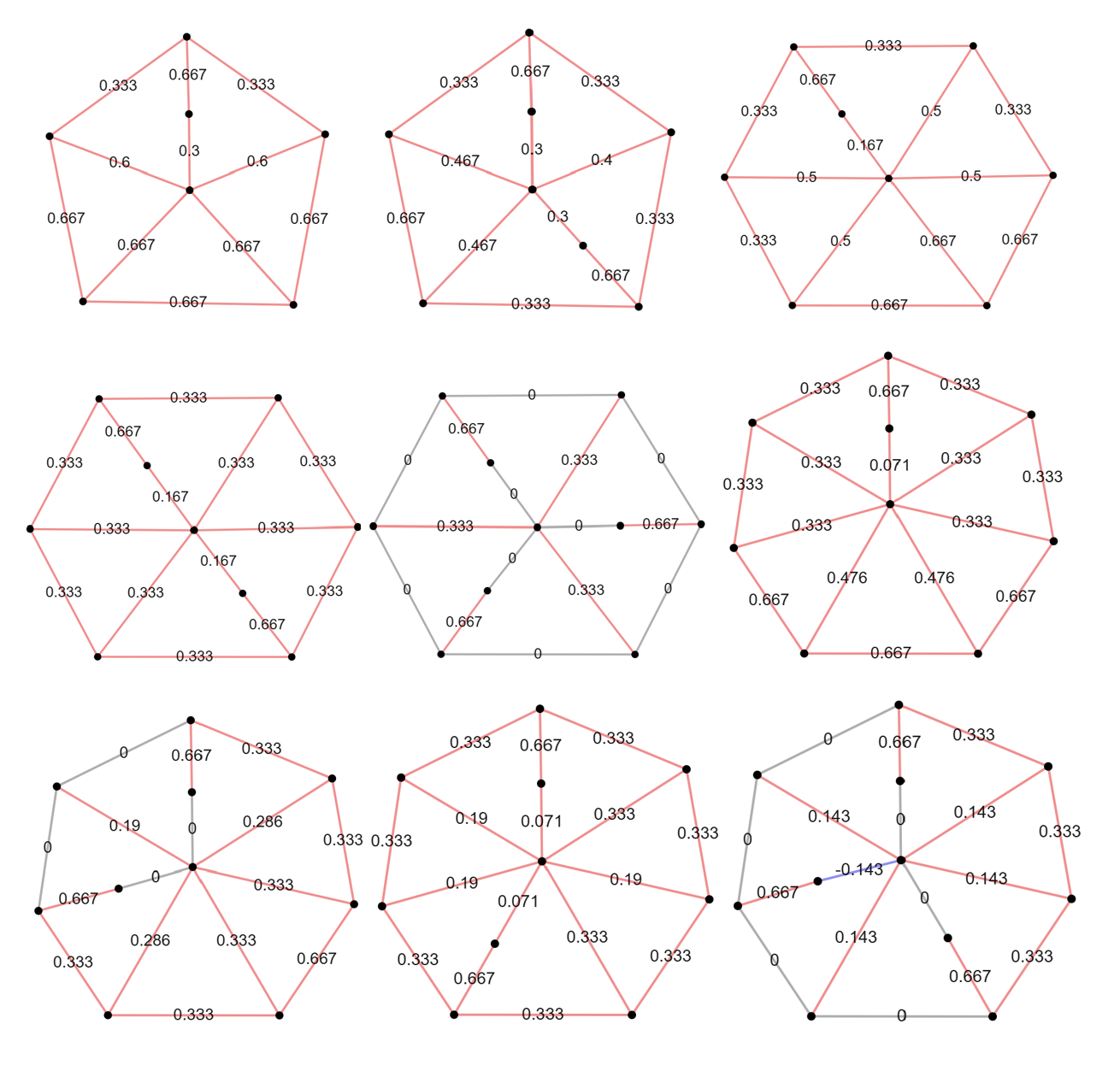}
   \caption{Remaining graphs in Subcase 1.2}
   \label{graph subcase 1.2}
\end{figure}

\begin{figure}[ht]
  \centering
 \includegraphics[width=0.8\textwidth]{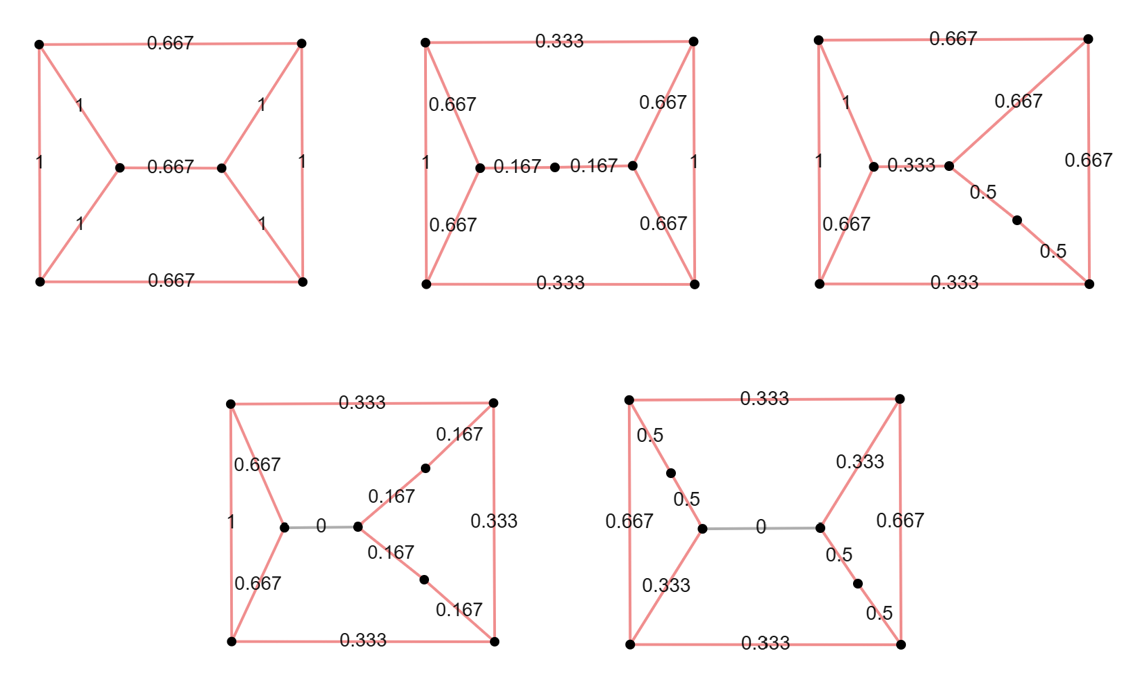}
   \caption{Remaining graphs in Subcase 2.2}
   \label{graph subcase 2.2}
\end{figure}

\end{document}